\documentclass[12pt,reqno]{amsart}
\usepackage{amssymb}
\usepackage{amsmath}
\usepackage{mathrsfs}
\usepackage{amssymb, amsmath, amsfonts}
\usepackage{extarrows}
\usepackage{mathabx}

\usepackage{color}

\usepackage[hidelinks]{hyperref}

\makeatletter
\def\tank#1{\protected@xdef\@thanks{\@thanks
 \protect\footnotetext[0]{#1}}}
\def\bigfoot{

 \@footnotetext}
\makeatother

\newcommand{\ea}{\end{array}}

\allowdisplaybreaks

\newtheorem{theorem}{Theorem}[section]
\newtheorem{lemma}{Lemma}[section]

\newtheorem{corollary}[theorem]{Corollary}
\newtheorem{definition}[theorem]{Definition}

\newtheorem{rem}{Remark}[section]

\def\beq{\begin{equation}}
\def\nneq{\end{equation}}

\def\bthm{\begin{theorem}}
\def\nthm{\end{theorem}}

\def\blem{\begin{lemma}}
\def\nlem{\end{lemma}}
\def\bprf{\begin{proof}}
\def\nprf{\end{proof}}
\def\bprop{\begin{prop}}
\def\nprop{\end{prop}}
\def\brmk{\begin{rem}}
\def\nrmk{\end{rem}}

\def\bexa{\begin{exa}}
\def\nexa{\end{exa}}
\def\bcor{\begin{cor}}
\def\ncor{\end{cor}}

\title[A note on decompositions of the stochastic convolution]{A note on decompositions  of the stochastic convolution driven by a white-fractional Gaussian noise}

\author{Ran Wang}
\curraddr[Wang, R.]{School of Mathematics and Statistics,  Wuhan University, P.R. China \& Computational Science Hubei Key Laboratory, Wuhan University,  P.R. China}

\author{Shiling Zhang}
\curraddr[Zhang S.]{School of Mathematics and Statistics,  Wuhan University, Wuhan, 430072,  China}
\email{shilingzhang@whu.edu.cn}

\date{}
\begin{document}
\maketitle

\noindent{\bf Abstract}
Let $u = \{u(t, x); (t,x)\in \mathbb R_+\times \mathbb R\}$ be the solution to a linear stochastic heat equation driven by a Gaussian noise, which is a Brownian motion in time and a fractional  Brownian motion in space with Hurst parameter $H\in(0, 1)$.
For any given $x\in \mathbb R$ (resp. $t\in \mathbb R_+$), we show  a decomposition of  the stochastic process $t\mapsto u(t,x)$ (resp. $x\mapsto u(t,x)$) as the sum of a fractional Brownian motion with Hurst parameter $H/2$ (resp. $H$) and a stochastic  process with  $C^{\infty}$-continuous trajectories.   Some applications of those decompositions are discussed.

\vskip0.3cm

\maketitle


\section{Introduction}
 Consider the following  one-dimensional stochastic heat equation
 \begin{equation}\label{eq SPDE}
 \frac{\partial u}{\partial t}=\frac{\kappa}{2}\frac{\partial ^2 u}{\partial x^2}+\dot W, \ \ \ t\ge0, x\in \mathbb R,
 \end{equation}
with some initial condition $u(0,x) \equiv0$, $\dot W=\frac{\partial^2 W}{\partial t\partial x}$, where  $W$ is a centered Gaussian process with covariance given by
\begin{equation}\label{eq cov}
 \mathbb E[W(s,x)W(t,y)]=\frac12\left(|x|^{2H}+|y|^{2H}-|x-y|^{2H} \right)(s\wedge t), \ \
 \end{equation}
for any  $s,t\ge0, x, y\in\mathbb R$, with $H \in (0,1)$. That is, $W$ is a standard Brownian motion in time and a fractional Browinian motion (fBm for short) with Hurst parameter $H$ in space.

When $H=1/2$,  $\dot W$ is a space-time white noise and $u$ is the classical stochastic convolution, which has been understood very well  (see e.g.,  \cite{Walsh}).   Theorem 3.3 in \cite{K} tells us that    the stochastic process $t\mapsto u(t,x)$ (resp. $x\mapsto u(t,x)$) can be represented as the sum of a fBm with Hurst parameter $1/4$ (resp. $1/2$) and a stochastic  process with  $C^{\infty}$-continuous trajectories.  Hence,    locally  $t\mapsto u(t,x)$ (resp. $x\mapsto u(t,x)$ ) behaves as a fBm with Hurst parameter $1/4$ (resp. $1/2$), and it has the same   regularity  (such as  the H\"older continuity, the global and local moduli of continuities,  Chung-type law  of the iterated logarithm) as a fBm with Hurst parameter $1/4$ (resp. $1/2$). See
    Lei and Nualart \cite{LN} for  earlier related work.

When $H\in(1/2,1)$,   Mueller and Wu \cite{MW} used such a  decomposition  to study the critical dimension for hitting points for the fBm; Tudor and Xiao \cite{TX} studied  the sample regularities of the solution  for the fractional-colored stochastic heat equation by using this decomposition.

 When $H\in (0,1/2)$, the spatial covariance $\Lambda$, given in
$$
\mathbb E\left[\dot W(s,x)\dot W(t,y)\right]=\delta_0(t-s)\Lambda(x-y),
$$
  is a distribution, which fails to be positive. The study of stochastic partial  differential  equations with this kind of  noises lies outside the scope of application of the classical references (see, e.g., \cite{ Dal, PZ, DPZ}).
    It seems that the decomposition results in \cite{MW} and \cite{TX} are very hard to be extended to the case of $H\in (0,1/2)$.

 Recently,  the problems of the stochastic partial differential equation  driven by a fractional Gaussian noise  in space with $H\in(0,1/2)$  have attracted many authors' attention.
 For example,
Balan  et al. \cite{BJQ} studied the existence and uniqueness of  a mild solution  for   stochastic heat equation   with affine multiplicative fractional noise in space, that is, the diffusion coefficient is given by an affine function $\sigma(x)=ax+b$. They
  established  the H\"older  continuity of the solution  in \cite{BJQ2016}. The case of the general nonlinear coefficient $\sigma$, which has a Lipschitz derivative and satisfies $\sigma(0)=0$, has been studied in  Hu et al. \cite{HHLNT}.

 In this paper, we give   unified  forms of the decompositions for the stochastic convolution about both temporal and spatial variables when $H\in (0,1)$. That is,  for any given $x\in \mathbb R$ (resp. $t\in \mathbb R_+$), we show  a decomposition of  the stochastic process $t\mapsto u(t,x)$ (resp. $x\mapsto u(t,x)$) as the sum of a fractional Brownian motion with Hurst parameter $H/2$ (resp. $H$) and a stochastic  process with  $C^{\infty}$-continuous trajectories.
  Those decompositions not only lead to a better understanding of the   H\"older regularity of the stochastic convolution   \eqref{eq SPDE}, but also give the uniform and local moduli of continuities and Chung-type law of iterated logarithm.

 Notice that our decompositions are natural   extensions of  \cite[Theorem 3.3]{K}, and they are given in    different forms  with that obtained in \cite{MW} and \cite{TX}.

 The rest of this paper is organized as follows. In Section 2, we recall some results about the
Gaussian noise    and the stochastic convolution. The main results are given in Section 3, and their proofs are given in Section 4.

\section{The Gaussian noise and the stochastic  convolution}

 In this section, we introduce the Gaussian noise and corresponding stochastic integration, borrowed from  \cite{BJQ} and \cite{PT}.

 Let $\mathcal S$ be the space of Schwartz functions, and $\mathcal S'$ be its dual, the space of tempered distributions. The Fourier transform of a function $u\in \mathcal S$ is defined as
 $$
 \mathcal Fu(\xi)=\int_{\mathbb R}e^{-i\xi x}u(x)dx,
 $$
and the inverse Fourier transform is given by $\mathcal F^{-1}u(\xi)=(2\pi)^{-1}\mathcal Fu(-\xi)$.

Given a domain $G\subset \mathbb R^n$ for some $n\ge1$, let $\mathcal D(G)$ be the space of all real-valued infinitely differential functions with compact support on $G$. According to \cite[Theorem 3.1]{PT}, the noise $W$ can be represented by a zero-mean Gaussian family $\{W(\phi), \phi\in \mathcal D((0,\infty)\times \mathbb R)\}$ defined on a complete probability space $(\Omega, \mathcal F, \mathbb P)$, whose covariance is given by
\begin{equation}\label{eq cov 2}
 \mathbb E[W(\phi)W(\psi)]=c_{1, H}\int_{\mathbb R_+\times \mathbb R} \mathcal F\phi(s,\xi)\overline{\mathcal F\psi(s,\xi)}|\xi|^{1-2H}dsd\xi,\ \ \ \ H \in (0,1),
\end{equation}
for any $\phi ,\psi \in \mathcal D((0,\infty)\times \mathbb R)$, where the Fourier transforms $\mathcal F \phi, \mathcal F\psi$ are understood as Fourier transforms in space only, $\bar z$ is the conjugation of a complex number $z$ and
\begin{equation}\label{const c1H}
c_{1, H}=\frac{1}{2\pi} \Gamma(2H+1)\sin(H \pi).
\end{equation}

Let $\mathcal H$ be the Hilbert space obtained by completing $\mathcal D (\mathbb R)$ under the inner production,
\begin{equation}
\langle \phi, \psi\rangle_{\mathcal H}=
\left\{
\begin{aligned}\label{eq inn}
  &c_{2, H}^2\int_{\mathbb R^2} (\phi( x+y)-\phi(x))(\psi( x+y)-\psi(x))|y|^{2H-2} dxdy,& H&\in(0,1/2);\\
  & \int_{\mathbb R} \phi(x)\psi(x)dx, & H&=1/2;\\
  &c_{3, H}^2\int_{\mathbb R^2} \phi( x+y)\psi(x)|y|^{2H-2} dxdy, &H&\in(1/2,1),
\end{aligned}
\ \right.
\end{equation}
for any $\phi ,\psi \in \mathcal H $,
where
$c_{2, H}^2=H\left(\frac12-H\right)$,
$c_{3, H}^2=H(2H-1).$
Denote $\|\phi\|_{\mathcal H}:=\sqrt{\langle \phi, \phi\rangle_{\mathcal H}}$ for any $\phi\in \mathcal H$.
Then
\begin{align}\label{eq cov 3}
 \mathbb E[W(\phi)W(\psi)]
 =\mathbb E\left[\int_{\mathbb R_+}\langle \phi(s), \psi(s)\rangle_{\mathcal H} ds\right].
\end{align}
See e.g., \cite{HHLNT, PT, TTV}.

Let
\begin{equation}\label{eq p}
p_t(x)=\frac{1}{\sqrt{2\pi \kappa t}}e^{-\frac{x^2}{2\kappa t}}
\end{equation}
be the heat kernel on the real line related to $\frac{\kappa}{2}\Delta$.
\begin{definition}\label{def solut}
  We say that a random field $u=\{
 u(t, x); t\in [0, T], x\in \mathbb R\}$ is a  mild  solution of \eqref{eq SPDE}, if $u$ is
predictable and for any $(t, x)\in  [0, T]\times \mathbb R$,
\begin{equation}\label{eq solut}
u(t,x)= \int_0^t\int_{\mathbb R}p_{t-s}(x-y)W(ds,dy),\ \ \ a.s..
\end{equation}
 It is usually called the {\it stochastic convolution}. Denote $u(t,x)$ by $u_t(x)$ for any $(t,x)\in \mathbb R_+\times\mathbb R$.
\end{definition}

\section{Main results}

Recall that a mean-zero Gaussian process $\{X_t\}_{t\ge0}$ is called a  (two-sided) fractional Brownian motion     with    Hurst parameter $H\in (0,1)$, if it satisfies
\begin{align}\label{eq fBm}
X_0=0, \ \ \ \mathbb E\left(|X_t-X_s|^2\right)=|t-s|^{2H}, \ \ \ t, s\in \mathbb R.
\end{align}

\begin{theorem}\label{thm main} For $H\in (0,1)$, the following results hold for the stochastic convolution $u=\{u_t(x); (t,x)\in \mathbb R_+\times \mathbb R\}$ given by \eqref{eq solut}:
\begin{itemize}
  \item[(a)]  For every $x\in \mathbb R$, there exists a fBm $\{X_t\}_{t\ge0}$ with  Hurst parameter $H/{2}$, such that
$$
u_t(x)- C_{1, H, \kappa} X_t, \ \ \ \ t\ge0,
$$
defines a mean-zero Gaussian process with a version that is continuous on $\mathbb R_+$ and infinitely differentiable on $(0, \infty)$, where
\begin{equation}\label{eq c} C_{1, H, \kappa}:=\left(  \frac{2^{1-H}\Gamma(2H)}{\kappa^{1-H}\Gamma(H)} \right)^{\frac12}.
\end{equation}
  \item[(b)]  For every $t>0$, there exists a fBm $\{B(x)\}_{x\in\mathbb R}$   with  Hurst parameter $H$,
 such that
 $$
 u_t(x)- \kappa^{-\frac12} B(x),\ \ \ \ x\in\mathbb R,
 $$
 defines a Gaussian random field with a version that is continuous on $\mathbb R$ and infinitely differentiable on $\mathbb R$.
\end{itemize}
\end{theorem}

Let us observe that Theorem \ref{thm main} says that, locally $t\mapsto u_t(x)$ behaves as a fBm with Hurst parameter $H/2$ and $x\mapsto u_t(x)$ behaves as  a fBm with Hurst parameter $H$. Thus, for instance, it follows from this observation and known facts about fBm (see e.g., \cite{MRm}, \cite[Chapter 1]{Mis},   or \cite{Xiao2008}), we can obtain the following   sample regularities  of the stochastic convolution.

By applying Theorem \ref{thm main} and   the H\"older continuity  result of fBms \cite[Chapter 1]{Mis}, we have the following well-known results, (see e.g.,  \cite[Chapter 3]{Walsh},  \cite[Theorem 1.1]{BJQ2016}).
\begin{corollary}
\begin{itemize}
  \item[(a)]
 For every $x\in \mathbb R$, the stochastic process
   $t\mapsto u_t(x)$ is a.s. H\"older continuous of parameter $H/2-\varepsilon$ for every $\varepsilon >0$.
     \item[(b)]
 For every $t>0$, the stochastic process
   $x\mapsto u_t(x)$ is a.s. H\"older continuous of parameter $H-\varepsilon$ for every $\varepsilon >0$.
\end{itemize}
 \end{corollary}

By applying Theorem \ref{thm main} and   the variations of fBms (see e.g.,\cite{MRm}), we have the following results.
\begin{corollary}\label{prop sample regularity11} Let $\mathcal N$ be a standard normal random variable. Then, for every $x\in \mathbb R$ and $[a,b]\subset \mathbb R_+$,
   \begin{align}\label{eq  variation 1}
    \lim_{n\rightarrow \infty}\sum_{a2^n\le i\le b 2^n}\left[u_{(i+1)/2^n}(x)-u_{i/2^n}(x) \right]^{\frac 2 H}= (b-a)C_{1, H, \kappa}^{2/H}\mathbb E\left[|\mathcal N|^{2/H}\right], \ \   a.s.;
\end{align}
  and for every $t>0$, $[c,d]\subset \mathbb R$,
   \begin{align}\label{eq  variation 2}
    \lim_{n\rightarrow \infty}\sum_{c2^n\le i\le d 2^n}\left[u_{t}((i+1)/2^n)-u_{t}(i/2^n) \right]^{\frac 1 H}= (d-c) C_{2, H, \kappa}^{1/H} \mathbb E\left[|\mathcal N|^{1/H}\right], \ \   a.s..
\end{align}
\end{corollary}

By applying Theorem \ref{thm main} and the global and local moduli of continuity results for fBms  (see e.g., \cite[Chapter 7]{MRm}), we have
\begin{corollary}\label{prop sample regularity2 }
\begin{itemize}
  \item[(a)](Global moduli of continuity for intervals).
  For every $x\in \mathbb R$ and $[a,b]\subset \mathbb R_+$, we have
  \begin{align}\label{eq LIL 1}
   \lim_{\varepsilon \rightarrow 0+}\sup_{s,t\in[a,b], 0<|t-s|\le \varepsilon}\frac{|u_{t}(x)-u_{s}(x)|}{ |t-s|^{H/2}\sqrt{2\ln  (1/|t-s|)} }= C_{3, H, \kappa},\ \ \     a.s.;
\end{align}
  and for every $t>0$, $[c,d]\subset \mathbb R$, we have
  \begin{align}\label{eq LIL 2}
   \lim_{\varepsilon \rightarrow 0+}\sup_{x,y\in[c,d], 0<|x-y|\le \varepsilon}\frac{|u_{t}(x)-u_{t}(y)|}{ |x-y|^{H}\sqrt{2\ln  (1/|x-y|)} }= C_{4, H, \kappa},\ \ \     a.s..
\end{align}

  \item[(b)](Local moduli of continuity for intervals).
  For every $t>0$ and $x\in \mathbb R$ , we have
   \begin{align}\label{eq local moduli 1}
   \varlimsup_{\varepsilon \rightarrow 0+}\frac{\sup_{0<|t-s|\le \varepsilon} |u_{t}(x)-u_{s}(x)|}{\varepsilon ^{H/2}\sqrt{2\ln\ln (1/\varepsilon )} }= C_{5, H, \kappa},\ \ \     a.s.;
   \end{align}
  and
   \begin{align}\label{eq local moduli 2}
   \varlimsup_{\varepsilon \rightarrow 0+}\frac{\sup_{0<|x-y|\le \varepsilon} |u_{t}(x)-u_{t}(y)|}{\varepsilon ^{H}\sqrt{2\ln\ln (1/\varepsilon)} }= C_{6, H, \kappa},\ \ \     a.s..
   \end{align}

\end{itemize}

\end{corollary}

By applying Theorem \ref{thm main} and the Chung-type law of iterated logarithm in \cite [Theorem 3.3]{MR}, we have
\begin{corollary}\label{prop sample regularity2 }

For every $t>0$ and $x\in \mathbb R$, we have
 \begin{align}\label{eq CLIL 1}
   \varliminf_{\varepsilon \rightarrow 0+}\frac{\sup_{0<|t-s|\le \varepsilon} |u_{t}(x)-u_{s}(x)|}{\left(\varepsilon/  \ln\ln (1/\varepsilon)\right)^{H/2} }= C_{7, H, \kappa},\ \ \     a.s.;
\end{align}
and
 \begin{align}\label{eq CLIL 1}
   \varliminf_{\varepsilon \rightarrow 0+}\frac{\sup_{0<|x-y|\le \varepsilon} |u_{t}(x)-u_{t}(y)|}{\left(\varepsilon/  \ln\ln (1/\varepsilon)\right)^{H} }= C_{8, H, \kappa},\ \ \     a.s..
\end{align}

\end{corollary}

\section{The proof of Theorem \ref{thm main}}

\subsection{The proof  of  (a) in Theorem \ref{thm main}}

The method of proof is similar to those of \cite[Theorem 3.3]{K} and \cite[Proposition 3.1]{FKM},  but is   complicated in the   fractional noise case.

Choose and fix some $x\in \mathbb R$. For every $t,\varepsilon >0$,  by \eqref{eq solut}, we have
\begin{align*}
&u_{t+\varepsilon}(x)-u_t(x)\notag\\
=& \int_0^t\int_{\mathbb R}[p_{t+\varepsilon-s}(x-y)-p_{t-s}(x-y)] W(ds, dy)+ \int_t^{t+\varepsilon}\int_{\mathbb R} p_{t+\varepsilon-s}(x-y)  W(ds, dy).
\end{align*}
 Let
\begin{align*}
 J_1:=&\int_0^t\int_{\mathbb R}[p_{t+\varepsilon-s}(x-y)-p_{t-s}(x-y)] W(ds, dy),\\
  J_2:=&\int_t^{t+\varepsilon}\int_{\mathbb R} p_{t+\varepsilon-s}(x-y)  W(ds, dy).
\end{align*}
The construction of the Gaussian noise $W$, which is white in time, ensures that $J_1$ and $J_2$ are independent mean-zero Gaussian random variables. Thus,
\begin{align*}
\mathbb E\left[|u_{t+\varepsilon}(x)-u_t(x)|^2 \right]
= \mathbb E(J_1^2)+\mathbb E(J_2^2).
\end{align*}
Let us compute their variances respectively.   First, we compute the variance of $J_2$ by noting that
\begin{align*}
\mathbb E(J_2^2)=&c_{1, H}\int_t^{t+\varepsilon} \int_{\mathbb R}\left|\mathcal F p_{t+\varepsilon-s}(\xi) \right|^{2}|\xi|^{1-2H}ds d\xi\\
=&   c_{1, H}   \int_t^{t+\varepsilon} \int_{\mathbb R} e^{-\kappa ( t+\varepsilon-s) |\xi|^2}|\xi|^{1-2H}ds d\xi\\
=&   c_{1, H}   \int_0^{\varepsilon} \int_{\mathbb R} e^{-\kappa s |\xi|^2}|\xi|^{1-2H}ds d\xi.
\end{align*}
The change of variables $\tau:=\sqrt{\kappa s}\xi$ yields
\begin{align}\label{eq J2}
\mathbb E(J_2^2)=&    c_{1, H}   \int_0^{\varepsilon} \int_{\mathbb R} e^{-|\tau|^2}|\tau|^{1-2H} (\kappa s)^{-1+H}  ds d\tau\notag\\
=& c_{1, H}\Gamma(1-H)   H^{-1}\kappa^{H-1}\varepsilon^{H}.
\end{align}
For the term $J_1$,   we have
\begin{align}\label{eq J1}
\mathbb E(J_1^2)=& c_{1, H}\int_0^{t} \int_{\mathbb R}\left|\mathcal F p_{t+\varepsilon-s}(\xi)- \mathcal F p_{t-s}(\xi)\right|^{2}|\xi|^{1-2H}ds d\xi\notag\\
=& c_{1, H}\int_0^{t} \int_{\mathbb R}\left|\mathcal F p_{s+\varepsilon}(\xi)- \mathcal F p_{s}(\xi)\right|^{2}|\xi|^{1-2H}ds d\xi\notag\\
=&c_{1, H}\int_0^{t} \int_{\mathbb R}e^{-\kappa s|\xi|^2}\left(1-e^{-\frac{\kappa \varepsilon |\xi|^2}{2}}\right)^2 |\xi|^{1-2H}ds d\xi.
\end{align}
This  integral is hard to evaluate. By the change of variables and Lemma \ref{lem integ} in appendix, we have
\begin{align}\label{eq J12}
\int_0^{\infty} \int_{\mathbb R}e^{-\kappa s |\xi|^2}\left(1-e^{-\frac{\kappa \varepsilon |\xi|^2}{2}}\right)^2 |\xi|^{1-2H}ds d\xi =& \kappa^{-1}\int_{\mathbb R} \left(1-e^{-\frac{\kappa \varepsilon  |\xi|^2}{2}}\right)^2 |\xi|^{-1-2H}  d\xi\notag\\
=&   \varepsilon^H \kappa ^{H-1}\int_{\mathbb R}\left(1-e^{-\frac{  |\tau|^2}{2}}\right)^2 |\tau|^{-1-2H}  d\tau\notag\\
=&\Gamma(1-H) H^{-1}(2^{1-H}-1)\kappa^{H-1}\varepsilon^H.
\end{align}
Therefore,
\begin{align}\label{eq J13}
\mathbb E(J_1^2)=& c_{1, H} \Gamma(1-H) H^{-1}(2^{1-H}-1)\kappa^{H-1}\varepsilon^H\notag\\
&-   c_{1, H}\int_t^{\infty} \int_{\mathbb R}e^{-\kappa s |\xi|^2}\left(e^{-\frac{\kappa \varepsilon |\xi|^2}{2}}-1\right)^2 |\xi|^{1-2H}ds d\xi,
\end{align}
and hence
\begin{align}\label{eq J3}
 \mathbb E(J_1^2)+\mathbb E(J_2^2)= &c_{1, H}\Gamma(1-H)   H^{-1} 2^{1-H} \kappa^{H-1}\varepsilon^{H}\notag\\
 &-  c_{1, H}\int_t^{\infty} \int_{\mathbb R}e^{-\kappa s |\xi|^2}\left(1-e^{-\frac{\kappa \varepsilon |\xi|^2}{2}}\right)^2 |\xi|^{1-2H}ds d\xi.
\end{align}
Let $\eta$ denote a  white noise  on $\mathbb R$ independent of $W$, and consider the Gaussian process $\{T_t\}_{t\ge0}$ defined by
$$
T_t:=\left(\frac{c_{1,H}}{\kappa}\right)^{\frac12}\int_{-\infty}^{\infty} \left( 1- e^{-\frac{ \kappa t |\xi|^2}{2}}\right) |\xi|^{-\frac12-H}\eta(d\xi),\   \ t\ge0.
$$
Then $\{T_t\}_{t\ge0}$ is a well-defined mean-zero Wiener integral process, $T_0=0$, and
$$
{\rm Var} (T_t)=\frac{c_{1,H}}{\kappa}\int_{-\infty}^{\infty} \left( 1- e^{-\frac{ \kappa t |\xi|^2}{2}}\right)^2 |\xi|^{-1-2H}d\xi<\infty,  \ \ \text{for all } t>0.
$$
Furthermore, we note that for any $t,\varepsilon>0$,
\begin{align}\label{eq T}
\mathbb E(|T_{t+\varepsilon}-T_t|^2)=&\frac{c_{1,H}}{\kappa} \int_{-\infty}^{\infty} \left(   e^{-\frac{ \kappa t |\xi|^2}{2}}- e^{-\frac{ \kappa (t+\varepsilon) |\xi|^2}{2}}\right)^2 |\xi|^{-1-2H}d\xi\notag\\
=&\frac{c_{1,H}}{\kappa}\int_{-\infty}^{\infty}  e^{- \kappa t |\xi|^2}  \left( 1 - e^{-\frac{ \kappa \varepsilon |\xi|^2}{2}}\right)^2 |\xi|^{-1-2H}d\xi\notag\\
=&  c_{1, H}  \int_t^{\infty} \int_{-\infty}^{\infty} e^{- \kappa s |\xi|^2}  \left( 1 - e^{-\frac{ \kappa \varepsilon |\xi|^2}{2}}\right)^2 |\xi|^{1-2H} dsd\xi.
 \end{align}
This is precisely the missing integral in \eqref{eq J3}. Therefore, by the independence of $T$ and $u$, we can rewrite \eqref{eq J3} as follows:
\begin{align}
&\mathbb E\left(\left|(u_{t+\varepsilon}(x)+T_{t+\varepsilon})-(u_t(x)+T_t)\right|^2\right)\notag\\
=&c_{1, H}\Gamma(1-H)   H^{-1}\kappa^{H-1} 2^{1-H}\varepsilon^{H}\notag\\
=&\frac{ \sin( H \pi  )\Gamma(2H+1)\Gamma(1-H) }{2^H H  \kappa^{1-H} \pi} \varepsilon^{H}\notag\\
=&  \frac{2^{1-H}\Gamma(2H)}{\kappa^{1-H}\Gamma(H)}\varepsilon^{H},
\end{align}
This implies that
$$
X_t:=\left(\frac{2^{1-H}\Gamma(2H)}{\kappa^{1-H}\Gamma(H)}\right)^{-\frac12}(u_t(x)+T_t),  \  \ \ t\ge0,
$$
is a fBm with   Hurst parameter $H/2$. Using the same argument in the proof of \cite [Lemma 3.6]{K}, we know that the random process $\{T(t)\}_{t\ge0}$ has a version that is   infinitely-differentiable on $(0, \infty)$.

\subsection{The proof  of  (b) in Theorem \ref{thm main}} This result has been proved in \cite[Proposition 3.1]{FKM} when $H=1/2$. We will prove it for the case of $H\neq 1/2$.

 For any $t>0, x\in \mathbb R$,  let
  \begin{equation}\label{eq S}
    S_t(x):=\int_t^{\infty} \int_{\mathbb R} \left[p_s(x-w)-p_s(w)\right]\zeta(ds,dw),
  \end{equation}
 where $p_t(x)$ is given by \eqref{eq p},  $\zeta$ is a Gaussian noise independent with $W$, which is white in time and fractional in space variable   with Hurst parameter $H$.
By the   argument  in the proof of \cite [Lemma 3.6]{K}, we know that $\{S_t(x)\}_{x\in \mathbb R}$ admits a $C^{\infty}$-version for any $t>0$.

Next, we will prove that
\begin{equation}\label{eq claim}
\mathbb E\left[\left|\left(u_t(x+\varepsilon)+S_t(x+\varepsilon) \right)-\left(u_t(x)+S_t(x) \right) \right|^2\right]= \kappa^{-1}\varepsilon^{2H}.
\end{equation}
Then
\begin{equation}\label{aim}
B(x):= \kappa^{1/2}(u_t(x)+S_t(x)),\ \ \ \ x\in\mathbb R,
\end{equation}
is a two-side fBm with parameter $H$, and (b) in Theorem \ref{thm main} holds.

 In the remaining part, we will prove \eqref{eq claim} for $H\in (0,1/2)$ and  $H\in (1/2,1)$, respectively.

\subsubsection{$(0<H<1/2)$}
  For any fixed $t>0$ and $\varepsilon>0$, by Plancherel's identity with respect to $y$ and the explicit formula for $\mathcal F p_t$, we have
\begin{align}
&\mathbb E\left(|u_t(x+\varepsilon)-u_t(x)|^2 \right) \notag \\
=&c_{2, H}^2\int_0^t\int_{\mathbb R^2}\Big[(p_{t-s}(x+\varepsilon-y+z)-p_{t-s}(x-y+z))-(p_{t-s}(x+\varepsilon-y)-p_{t-s}(x-y))\Big]^2 \notag \\
& \ \ \ \ \ \ \ \ \ \ \ \ \ \ \ \ \times |z|^{2H-2}dsdydz\notag\\
=& \frac{1}{2\pi}c_{2,H}^2 \int_0^t \int_{\mathbb R^2} e^{-\kappa (t-s) |\xi|^2}\left|e^{i\xi z}-1\right|^2\left|e^{i \xi \varepsilon}-1\right|^2 |z|^{2H-2}ds d\xi dz.
\end{align}
Since $|e^{i\xi z}-1|^2=2(1-\cos (\xi z))$  and   for any $\alpha\in(0,1)$,
\begin{equation}\label{eq iden1}
\int_0^{\infty} \frac{1-cos(\xi z)}{z^{1+\alpha}}dz= \alpha^{-1}\Gamma(1-\alpha)\cos(\pi\alpha/2) \xi^{\alpha},
\end{equation}
(see  \cite [Lemma D.1]{BJQ}), we have
\begin{align}\label{uoff}
&\mathbb E\left(|u_t(x+\varepsilon)-u_t(x)|^2 \right) \notag \\
=&\frac{2 \Gamma(2H)\sin( H \pi)}{(1-2H) \pi}c_{2, H}^2 \int_0^t \int_{\mathbb R} e^{-\kappa (t-s)|\xi|^2}
\left|e^{i \xi \varepsilon}-1\right|^2 |\xi|^{1-2H}dsd\xi  \notag \\
=&\frac{4\Gamma(2H)\sin( H \pi)}{(1-2H) \kappa\pi }c_{2, H}^2  \int_{\mathbb R}\left(1- e^{- \kappa t  |\xi|^2}\right)\left(\frac{1-\cos (\xi \varepsilon)}{|\xi|^{1+2H}}\right)d\xi \notag \\
=& \frac{4\Gamma(2H)\sin( H \pi)}{(1-2H) \kappa\pi }c_{2, H}^2 \notag \\
&\ \ \ \times\left(\frac{\Gamma(1-2H)\cos( H \pi)\varepsilon^{2H}}{H}   - \int_{\mathbb R} e^{- \kappa t |\xi|^2} \left(\frac{1-\cos (\xi \varepsilon)}{|\xi|^{1+2H}}\right)d\xi \right).
\end{align}
Recall $S_t(x)$ defined by \eqref{eq S}. Using the same techniques above, we have
\begin{align}
& \mathbb E\left(|S_t(x+\varepsilon)-S_t(x)|^2 \right) \notag \\
=& \frac{2\Gamma(2H)\sin( H \pi)}{ (1-2H)\pi}c_{2, H}^2 \int_t^{\infty} \int_{\mathbb R} e^{-\kappa s |\xi|^2}
\left|e^{i \xi \varepsilon}-1\right|^2 |\xi|^{1-2H}ds d\xi\notag \\
=&\frac{4\Gamma(2H)\sin( H \pi)}{(1-2H) \kappa\pi }c_{2, H}^2 \int_{\mathbb R} e^{- \kappa t |\xi|^2} \left(\frac{1-\cos (\xi \varepsilon)}{|\xi|^{1+2H}}\right)d\xi,
\end{align}
which is exactly the missing integral in \eqref{uoff}.  By the independence of $W$ and $\zeta$, we know that $u_t(x)$ and $S_t(x)$ are independent. Therefore, we have
\begin{equation}
\begin{aligned}
&\mathbb E\left[\left|\left(u_t(x+\varepsilon)+S_t(x+\varepsilon) \right)-\left(u_t(x)+S_t(x) \right) \right|^2\right]\\
=& \frac{\sin(2 H \pi)\Gamma(2H)\Gamma(1-2H) }{\kappa \pi }
 \varepsilon^{2H}\\
 =&  \kappa^{-1}  \varepsilon^{2H}.
\end{aligned}
\end{equation}
\quad

\subsubsection{$(1/2<H<1)$}
For any fixed $t>0$ and $\varepsilon>0$,
\begin{align}
&\mathbb E\left(|u_t(x+\varepsilon)-u_t(x)|^2 \right) \notag \\
=&c_{3, H}^2\int_0^t\int_{\mathbb R^2} (p_{t-s}(x+\varepsilon-y+z)-p_{t-s}(x-y+z))(p_{t-s}(x+\varepsilon-y)-p_{t-s}(x-y)) \notag \\
& \ \ \ \ \ \ \ \ \ \ \ \ \ \ \ \ \ \ \times |z|^{2H-2}dsdydz.
\end{align}
Since $p_{t-s}(x+y)p_{t-s}(x)=\frac12\left[p^2_{t-s}(x+y)+p^2_{t-s}(x)-(p_{t-s}(x+y)-p_{t-s}(x))^2\right]$,  by Plancherel's identity with respect to $x$ and the explicit formula for $\mathcal F p_t$, we have
$$ \int_{\mathbb R}\int_{\mathbb R} p_{t-s}(x+y)p_{t-s}(x)|y|^{2H-2} dxdy = \frac{1}{2\pi}\int_{\mathbb R}\int_{\mathbb R} e^{-\kappa (t-s) |\xi|^2} \cos(\xi y) |y|^{2H-2} d\xi dy.$$
Therefore,
\begin{align}
&\mathbb E\left(|u_t(x+\varepsilon)-u_t(x)|^2 \right) \notag \\
=&\frac{1}{2\pi}c_{3, H}^2\int_0^t\int_{\mathbb R^2}e^{-\kappa (t-s) |\xi|^2}\left[ 2\cos(\xi z)-\cos(\xi(\varepsilon+z))-\cos(\xi(\varepsilon-z))\right]|z|^{2H-2}ds d\xi dz \notag \\
=&\frac1\pi c_{3, H}^2\int_0^t\int_{\mathbb R^2}e^{-\kappa (t-s) |\xi|^2}\cos(\xi z)(1-\cos(\xi \varepsilon))|z|^{2H-2}ds d\xi dz.
\end{align}
By formula $\left( 3.761-9 \right)$  of \cite {GR}, we know that
\begin{equation}\label{eq formula}
\int_0^{\infty} \frac{\cos(ax)}{x^{1-\mu}}dx=\frac{\Gamma(\mu)}{a^\mu}\cos(\pi\mu/2),\ \ \   \text{for any} \ \mu \in (0,1), a>0.
\end{equation}
Since $H \in (1/2,1)$, using \eqref{eq formula} with $\mu=2H-1,$ we have
\begin{align}
&\mathbb E\left(|u_t(x+\varepsilon)-u_t(x)|^2 \right) \notag \\
=&\frac{2\Gamma(2H-1)\cos(\pi(2H-1)/2)}{\pi}c_{3, H}^2\int_0^t\int_{\mathbb R}e^{-\kappa (t-s)  |\xi|^2} (1-\cos(\xi \varepsilon) |\xi|^{1-2H}ds d\xi \notag \\
=&\frac{2\Gamma(2H-1)\sin( H \pi)}{\kappa \pi}c_{3, H}^2 \int_{\mathbb R}\left(1- e^{- \kappa t  |\xi|^2}\right)\left(\frac{1-\cos (\xi \varepsilon)}{|\xi|^{1+2H}}\right)d\xi \notag \\
=&\frac{2\Gamma(2H-1)\sin( H \pi)}{\kappa \pi}c_{3, H}^2 \notag \\
&\ \ \ \times \left( -\frac{\Gamma(2-2H)\cos( H \pi)}{H(2H-1)}\varepsilon^{2H} - \int_{\mathbb R} e^{- \kappa t |\xi|^2} \left(\frac{1-\cos (\xi \varepsilon)}{|\xi|^{1+2H}}\right)d\xi \right),
\end{align}
where the last equality we used the identity:
\begin{equation}
\int_0^{\infty} \frac{1-cos(\xi z)}{z^{1+\alpha}}dz= -\alpha^{-1}(\alpha-1)^{-1}\Gamma(2-\alpha)\cos(\pi\alpha/2) \xi^{\alpha},\ \ \ \ \alpha\in(1,2),
\end{equation}
(see  \cite [Lemma D.1]{BJQ}).
Recall $S_t(x)$  given by  \eqref{eq S}. Using the same techniques above, we have
\begin{align}
& \mathbb E\left(|S_t(x+\varepsilon)-S_t(x)|^2 \right) \notag \\
=& \frac{2\Gamma(2H-1)\sin( H \pi)}{\pi}c_{3, H}^2 \int_t^{\infty} \int_{\mathbb R} e^{-\kappa s |\xi|^2}
\left|e^{i \xi \varepsilon}-1\right|^2 |\xi|^{1-2H}ds d\xi\notag \\
=&\frac{2\Gamma(2H-1)\sin( H \pi)}{\kappa \pi}c_{3, H}^2 \int_{\mathbb R} e^{- \kappa t |\xi|^2} \left(\frac{1-\cos (\xi \varepsilon)}{|\xi|^{1+2H}}\right)d\xi.
\end{align}
By the independence of $W$ and $\zeta$, we know that $u_t(x)$ and $S_t(x)$ are independent. Therefore, we have
\begin{equation}
\begin{aligned}
&\mathbb E\left[\left|\left(u_t(x+\varepsilon)+S_t(x+\varepsilon) \right)-\left(u_t(x)+S_t(x) \right) \right|^2\right]\\
=& \frac{- \sin(2H \pi) \Gamma(2H-1) \Gamma(2-2H) }{\kappa \pi} \varepsilon^{2H}\\
=&\kappa^{-1} \varepsilon^{2H}.
\end{aligned}
\end{equation}

\section{Appendix}
\begin{lemma}\label{lem integ} The following identity holds:
$$\int_0^{\infty}\left(e^{-\frac{x^2}{2}}-1\right)^2 x^{-1-2H}  dx=  \Gamma(1-H) H^{-1}(2^{-H}-2^{-1}).$$
\end{lemma}
\begin{proof}
The proof is inspired by Lemma A.1 in \cite{K}. By the change of  variables $w=x^2/2$, we have
\begin{align*}
\int_0^{\infty}\left(e^{-\frac{x^2}{2}}-1\right)^2 x^{-1-2H}  dx
=& 2^{-1-H} \int_0^{\infty}\left(e^{-w}-1\right)^2 w^{-1-H}  dw\notag\\
=& 2^{-1-H} (I_{0,1}-I_{1,2}),
\end{align*}
where $I_{a, b}:=\int_0^{\infty}\left( e^{-aw}- e^{-b w}\right)  w^{-1-H}  dw$ for all $a,b\ge0$. Since  $e^{-aw}-e^{-b w}=w\int_a^b e^{-rw}dr$, we have
\begin{align*}
I_{a, b}=  \int_0^{\infty}\int_a^b e^{-rw} w^{-H}drdw
= \Gamma(1-H) \int_a^b r ^{-1+H}dr
= \Gamma(1-H) H^{-1} (b^H-a^H).
\end{align*}
Thus, $I_{0,1}-I_{1,2}=\Gamma(1-H) H^{-1} (2-2^H)$, and the lemma follows.
\end{proof}

\vskip0.5cm

\vskip0.5cm

\noindent{\bf Acknowledgments}:   R. Wang is supported by  NSFC (11871382), Chinese State Scholarship Fund Award by the China Scholarship Council  and  Youth Talent Training Program by Wuhan University.


\vskip0.5cm


\begin{thebibliography}{abc}

\bibitem{BJQ} Balan R, Jolis M,   Quer-Sardanyons L. SPDEs with affine multiplicative fractional noise in space with inder $\frac14<H<\frac12$. Electron J Probab, 2015, {\bf 20}(54): 1-36

\bibitem{BJQ2016} Balan R, Jolis M,  Quer-sardanyons L. SPDEs with rough noise in space: H\"older continuity of the solution. Statist Probab Lett, 2016, {\bf 119}: 310-316








\bibitem{Dal} Dalang R C. Extending martingale measure stochastic integral with applications to spatially homogeneous s.p.d.e's. Electron  J  Probab,  1999, {\bf 4}(6): 1-29


\bibitem{FKM} Foondun M,  Khoshnevisan D,  Mahboubi P. Analysis of the gradient of the solution to a stochastic heat equation via fractional Brownian motion. Stoch PDE: Anal Comp. 2005, {\bf 3}: 133-158



\bibitem{GR} Gradshteyn I,  Ryzhik I. Table of integrals, series and products. Academic Press, 2007

\bibitem{HHLNT}Hu Y, Huang J, L\^e K, Nualart D,  Tindel S. Stochastic heat equation with rough dependence in space. Ann Probab, 2017, {\bf 45}(6): 4561-4616


\bibitem{LN}Lei P,  Nualart D. A decomposition of the bifractional Brownian motion and some applications. Statist Probab Lett, 2009, {\bf 79}(5): 619-624



\bibitem{MRm} Marcus M, Rosen J. Markov processes, Gaussian processes and local times. Cambridge Studies in Advanced Mathematics. Cambridge University Press, 2006


\bibitem{Mis} Mishura.  Stochastic calculus for fractional Brownian motion and related processes. Lecture Notes in Math, 1929. Springer-Verlag, Berlin, 2008

\bibitem{MR} Monrad D,  Rootz\'en H. Small values of Gaussian processes and functional laws of the iterated logarithm. Probab  Theory Related Fields, 1995, {\bf 192}(91): 173-192

\bibitem{MW} Mueller C,  Wu, Z. Erratum: A connection between the stochastic heat equation and fractional Brownian motion and a simple proof of a result of Talagrand. Electron Commun Probab, 2012, {\bf 17}(8): 1-10

\bibitem{K} Khoshnevisan D. Analysis of stochastic partial differential equations. CBMS Regional Conference Series in Mathematics, 119. American Mathematical Society, 2014

\bibitem{PZ} Peszat S,  Zabczyk J. Stochastic evolution equations with a spatially homogeneous Wiener process. Stochastic Process Appl, 1997, {\bf 72}(2): 187-204


\bibitem{PT} Pipiras V, Taqqu M S. Integration questions related to fractional Brownian motion. Probab  Theory Related Fields, 2000, {\bf 291}: 251-291

\bibitem{DPZ} Prato G D, Zabczyk J. Stochastic equations in infinite dimensions. Second edition. Encyclopedia of Mathematics and its Applications. Cambridge University Press, 2014

\bibitem{TTV} Tindel S, Tudor C, Viens F. Stochastic evolution equations with fractional Brownian motion, Probab. Theory Related Fields, 2003, {\bf 127}: 186-204

\bibitem{TX} Tudor C, Xiao Y M. Sample paths of the solution to the fractional-colored stochastic heat equation. Stoch Dyn, 2017, {\bf 17}(1): 20pp


\bibitem{Walsh} Walsh J.  An introduction to stochastic partial differential equations. Lecture Notes in Mathematics, 1180. Springer-Verlag, Berlin, 1986


\bibitem{Xiao2008} Xiao Y M. Strong local nondeterminism and sample path properties of Gaussian random fields.   Asymptotic theory in probability and statistics with applications,
Adv Lect Math,  2008,  136-176

\end{thebibliography}
\end{document}